\documentclass[12pt,a4paper]{article}

\usepackage{amsfonts,amssymb,amsthm}
\usepackage{amsbsy}
\usepackage{graphicx}
\usepackage{latexsym}
\usepackage{amsmath,enumerate,amsfonts,amssymb,color,graphicx,amsthm}

\usepackage{CJK,CJKnumb}

\theoremstyle{plain}
\newtheorem{theorem}{Theorem}[section]

\newtheorem{lemma}[theorem]{Lemma}
\newtheorem{prop}[theorem]{Proposition}
\theoremstyle{definition}

\theoremstyle{remark}
\newtheorem{remark}{Remark}[section]
\numberwithin{equation}{section}

\theoremstyle{example}

\newcounter{marnote}

\def\o{\overline}
\def\u{\underline}

\def\b{\backslash}

\begin{document}

\begin{center}{\Large Interior derivative estimates and Bernstein theorem for Hessian quotient equations}
\end{center}
\centerline{\large  Limei Dai \quad Jiguang Bao \quad Bo Wang }
\vspace{5mm}

\begin{minipage}{140mm}{\footnotesize {\small\bf Abstract:} {\small
In this paper, we obtain the interior derivative estimates of solutions for elliptic and parabolic Hessian quotient equations. Then we establish the Bernstein theorem for parabolic Hessian quotient equations, that is, any parabolically convex solution $u=u(x,t)\in C^{4,2}(\mathbb{R}^n\times (-\infty,0])$ for $-u_t\frac{S_n(D^2u)}{S_l(D^2u)}=1$ in $\mathbb{R}^n\times (-\infty,0]$ must be the form of $u=-mt+P(x)$ with $m>0$ being a constant and $P$ being a convex quadratic polynomial.}}



\end{minipage}

\section{Introduction}


In this paper, we consider the derivative estimates and the Bernstein theorem for the Hessian quotient equations
\begin{equation}\label{hq-3}\frac{S_{k}(D^2u)}{S_{l}(D^2u)}=1\ \ \mbox{in}\ \ D,\end{equation}
and
\begin{equation} \label{hq1}-u_t\frac{S_k(D^2u)}{S_l(D^2u)}=1\
\ \mbox{in}\ \
\mathbb{R}^n\times(-\infty,0]
\end{equation}
where $0\leq l<k\leq
n,n\geq 2$, $S_l(D^2u)$
is the $l$th elementary symmetric function of the
eigenvalues $\lambda=(\lambda_1,\lambda_2,\cdots,\lambda_n)$ of Hessian matrix
$D^2u$, i.e.,
$$S_l(D^2u)=S_l(\lambda(D^2u))=\sum_{1\leq
i_1<\cdots<i_l\leq
n}\lambda_{i_1}\cdots\lambda_{i_l},l=1,2,\cdots,n,$$
and $D$ is a convex domain in $\mathbb{R}^n$.
For $l=0$, we set $S_0(D^2u)\equiv 1.$

J\"{o}rgens \cite{j} ($n=2$), Calabi \cite{c} ($n\leq 5$) and Pogorelov \cite{p} ($n\geq 2$) asserted that any convex classical solution of $\det D^2u=1$ in $\mathbb{R}^n$ must be a quadratic polynomial. Cheng and Yau \cite{cy} gave a simpler proof for the result. Caffarelli \cite{ca} extended the result to viscosity solutions. One can also refer to the related results \cite{jw, jx, tw1}. Later, the result was generalized to many equations, see \cite{BCGJ, bqtw, GH1, jw, lly, nt, tw1, xb} and the references therein. For instance, Bao, Chen, Guan and Ji \cite{BCGJ} obtained the result for elliptic Hessian quotient equations. Guti\'{e}rrez and Huang \cite{GH1} proved the Bernstein type theorem for parabolic Monge-Amp\`{e}re equations
\begin{equation}\label{pma}-u_t\det D^2u=1\ \mbox{in}\ \mathbb{R}^n\times(-\infty,0],\end{equation}
which is stated that if $u\in C^{4,2}(\mathbb{R}^n\times(-\infty,0])$ is a parabolically convex solution of \eqref{pma} such that for some positive constants $m_1,m_2$, $-m_2\leq u_t(x,t)\leq -m_1$, then $u$ must be the form $u(x,t)=-mt+P(x)$ with $m>0$ and $P$ being a convex quadratic polynomial. Xiong and Bao \cite{xb} obtained the Bernstein type theorem for parabolic Monge-Amp\`{e}re equations $u_t=\rho(\log \det D^2u)\ \mbox{in}\ \mathbb{R}^n\times(-\infty,0]$ with $\rho\in C^2(\mathbb{R})$. Nakamori and Takimoto \cite{nt} obtained the Bernstein type theorem of parabolically convex solutions for parabolic Hessian equations $-u_tS_k(D^2u)=1\ \mbox{in}\ \mathbb{R}^n\times(-\infty,0].$ Recently, He, Sheng and Xiang \cite{hsx} and Bao, Qiang, Tang and Wang \cite{bqtw} respectively proved the Bernstein type theorem of parabolic Hessian equations for parabolically $(k+1)-$convex solutions and parabolically $k-$convex solutions.

In this paper, we first obtain the Pogorelov estimate for elliptic Hessian quotient equation. The interior gradient estimate of elliptic Hessian quotient equations can be referred to \cite{chen}.

\begin{theorem}\label{thm-hqe} Let $k=n$ and $D$ be a bounded convex domain in $\mathbb{R}^n.$ Assume that $u\in C^{4}(\o{D})$ is a convex solution to \eqref{hq-3} with
\begin{equation}\label{hq-4}u=0\ \ \mbox{on}\ \ \partial D.\end{equation}
Then there exists a constant $C_0=C_0(n,||u||_{C^1(D)})$ such that
\begin{equation}\label{hq-2}\sup_{x\in D}|u(x)|^4|D^2u(x)|\leq C_0.\end{equation}
\end{theorem}

Let $\Omega\subset\mathbb{R}^n\times(-\infty,0]$ and $t\leq 0$, and let
$$\Omega(t)=\{x:(x,t)\in \Omega\}.$$
Suppose that $\Omega$ is a bounded domain and $\u{t}=\inf\{t:\Omega(t)\not=\emptyset\}.$ Let $\o{\Omega(\u{t})}$ denote the closure of $\Omega(\u{t})$ and $\partial \Omega(t)$ denote the boundary of $\Omega(t),$ we define the parabolic boundary of $\Omega$ as
$$\partial_p\Omega=(\o{\Omega(\u{t})}\times\{\u{t}\})\cup \bigcup_{t\in \mathbb{R},t\leq 0}(\partial \Omega(t)\times\{t\}).$$
The domain $\Omega\subset \mathbb{R}^{n}\times(-\infty,0]$ is called a bowl-shaped domain if for each $t$, $\Omega(t)$ is convex and for $t_1\leq t_2$, $\Omega(t_1)\subset \Omega(t_2)$. One can also refer to \cite{GH1}.

A function $u=u(x,t):\mathbb{R}^n\times(-\infty,0]\rightarrow \mathbb{R}$ is called a parabolically convex function if $u$ is strictly convex in $x$ and decreasing in $t$. Consider the equation
\begin{equation}\label{kl}-u_t\frac{S_{k}(D^2u)}{S_{l}(D^2u)}=1\ \ \mbox{in}\ \ \Omega,\end{equation}
\begin{equation}\label{kl-1}u=0\ \ \mbox{on}\ \ \partial_p\Omega.\end{equation}

Let $B_r=B_r(0)=\{x\in \mathbb{R}^n: |x|<r\}$ for $r>0$. Our results are as follows.

\begin{theorem}\label{gradient}
Suppose that $\Omega\subset B_{r}(0)\times (-\infty,0]$ is a bowl-shaped domain, $(0,0)\in \Omega$. Let $u\in C^{4,2}(\Omega)\cap C^{2,1}(\o{\Omega})$ be a parabolically $k-$convex solution of \eqref{kl} and \eqref{kl-1} satisfying for some positive constant $m_1$
$$u_t(x,t)\leq -m_1\ \ (x,t)\in \Omega.$$
Then
\begin{equation}\label{Du}
|Du(0,0)|\leq C \frac{\sup_{\Omega}u}{r},
\end{equation}
where $C$ is a positive constant depending only on $n,k,l,m_1$.
\end{theorem}


Suppose that $u\in C^{4,2}(\o{\Omega})$ and there exist constants $0<m_1\leq m_2$ such that for any $(x,t)\in\mathbb{R}^n\times(-\infty,0]$,
\begin{equation}\label{ut}-m_2\leq u_t(x,t)\leq -m_1.\end{equation}

\begin{theorem}\label{thm-ns} Let $k=n$ and $\Omega$ be a bounded bowl-shaped domain in $\mathbb{R}^n\times (-\infty,0].$ Assume that $u\in C^{4,2}(\o{\Omega})$ is a parabolically convex solution to the Dirichlet problem \eqref{kl} and \eqref{kl-1} and satisfies \eqref{ut} in $\Omega$.
Then there exists a constant $C_0=C_0(n,m_1,||u||_{C^1(\Omega)})$ such that
\begin{equation}\label{sec-est}\sup_{(x,t)\in \Omega}|u(x,t)|^4|D^2u(x,t)|\leq C_0.\end{equation}
\end{theorem}

\begin{remark}
For the proof of Pogorelov type estimate, the cases of $u_{nn}\geq \kappa u_{11}$ for some small constant $\kappa$ and  $u_{nn}\leq \kappa u_{11}$ are often divided, see \cite{bqtw, cw, nt}. Here for the proof of Theorems \ref{thm-hqe} and \ref{thm-ns}, we adopt the inequalities in \cite{chen} and do not divide the cases.
\end{remark}
Due to Theorems \ref{gradient} and \ref{thm-ns}, we have

\begin{theorem}\label{thm}
Let $k=n$ and $u\in C^{4,2}(\mathbb{R}^n\times(-\infty,0])$ be a parabolically convex solution to \eqref{hq1}.
Assume that \eqref{ut} holds and there exist positive constants $A_1, B$ such that for $x\in \mathbb{R}^n$,
\begin{equation}\label{u0} u(x,0)\leq A_1|x|^2+B.\end{equation}
 Then $u$ must be $u(x,t)=-mt+P(x)$, where $m>0$ is a constant and $P$ is a convex quadratic polynomial.
\end{theorem}

The remainders of this paper are arranged as follows. Theorem \ref{thm-hqe} will be proved in Section 2. In Section 3, we will reach the interior gradient estimate and the Pogorelov type estimate, i.e., Theorems \ref{gradient}-\ref{thm-ns}. In section 4, we will prove the Bernstein theorem, i.e., Theorem \ref{thm}.

\section{Proof of Theorem \ref{thm-hqe}.}

For $\lambda=(\lambda_1,\dots,\lambda_n)\in \mathbb{R}^n$, denote $$S_{l;i_1i_2\dots i_j}(\lambda)=S_l(\lambda)|_{\lambda_{i_1}=\lambda_{i_2}=\dots=\lambda_{i_j}=0}.$$
For any $i$, we have
$$\frac{\partial S_l(\lambda)}{\partial \lambda_i}=S_{l-1;i}(\lambda).$$
Next we recall some basic properties for  elementary symmetric functions, which can be found in \cite{BCGJ, chen, cw}.
\begin{prop}\label{prop1}Let $l=0,1,\dots,n$ and $1\leq i\leq n.$
\begin{equation}\label{2.15-1}S_l(\lambda)=S_{l-1;i}\lambda_i+S_{l;i},\end{equation}
\begin{equation}\label{2.16-1} S_{n-1;j}S_l-S_{l-1;j}S_n=S_{n-1;j}S_{l;j},\end{equation}
\begin{equation}\label{ele-1}\sum_{i=1}^{n} \lambda_iS_{l-1;i}(\lambda)=lS_l(\lambda).\end{equation}
\end{prop}

Let $S_{l;i}(M)$ denote the symmetric function $S_l(M)$ with the matrix $M$ deleting the $i-$row and the $i-$column.

\begin{prop}\label{prop2}
Assume that the matrix $M=(M_{ij})$ is diagonal, then
$$\frac{\partial S_l(M)}{\partial M_{ij}}=
\begin{cases}
S_{l-1;i}(M),\ &\mbox{if}\ i=j,\\
0,\ &\mbox{if}\ i\not=j.
\end{cases}$$
\end{prop}

\begin{lemma}\label{lem1}
Let $\lambda=(\lambda_1,\dots,\lambda_n)$ with $\lambda_i>0,i=1,2,\dots,n.$ Then for $1\leq j\leq n,1\leq l<n,$ we have
\begin{equation}\label{eq6}
S_{n-2;ij}S_l-S_nS_{l-2;ij}>0.
\end{equation}
\end{lemma}
\begin{proof}
According to the equalities $S_{n-1;ij}=0$ and \eqref{2.15-1},
then
\begin{eqnarray*}
\lambda_i(S_{n-2;ij}S_l-S_nS_{l-2;ij})&=&\lambda_iS_{n-2;ij}S_l-\lambda_iS_{l-2;ij}S_n\\
&=&(S_{n-1;j}-S_{n-1;ji})S_l-(S_{l-1;j}-S_{l-1;ji})S_n\\
&=&S_{n-1;j}S_l-S_{l-1;j}S_n+S_nS_{l-1;ji}.
\end{eqnarray*}
Owing to \eqref{2.16-1}, so
\begin{eqnarray*}
\lambda_i(S_{n-2;ij}S_l-S_nS_{l-2;ij})&=&S_{n-1;j}S_{l;j}+S_nS_{l-1;ji}>0.
\end{eqnarray*}
As a result, \eqref{eq6} holds.
\end{proof}

\begin{lemma}\label{lem2}
Let $\lambda=(\lambda_1,\dots,\lambda_n)$ with $\lambda_1\geq\dots\geq \lambda_n>0.$ Then
\begin{equation}\label{lem2-1}\frac{S_{n-1;n}S_{l;n}}{S_l^2}\geq \frac{1}{n-1}\left(\frac{S_{n-1;1}S_{l;1}}{S_l^2}+\dots+\frac{S_{n-1;n-1}S_{l;n-1}}{S_l^2}\right).\end{equation}
\end{lemma}

\begin{proof}
Because for $\lambda_1\geq \lambda_2\geq \dots\geq \lambda_n$, we have  $$S_{n-1;n}\geq S_{n-1;n-1}\geq \dots\geq S_{n-1;1}$$ and
$$S_{l;n}\geq S_{l;n-1}\geq \dots\geq S_{l;1},$$
so \begin{equation}\label{e4}\frac{S_{n-1;1}S_{l;1}}{S_l^2}\leq \frac{S_{n-1;n}S_{l;n}}{S_l^2},\dots,\frac{S_{n-1;n-1}S_{l;n-1}}{S_l^2}\leq \frac{S_{n-1;n}S_{l;n}}{S_l^2}.\end{equation}
Therefore,
$$(n-1)\frac{S_{n-1;n}S_{l;n}}{S_l^2}\geq \frac{S_{n-1;1}S_{l;1}}{S_l^2}+\dots+\frac{S_{n-1;n-1}S_{l;n-1}}{S_l^2},$$
which indicates that \eqref{lem2-1} holds.
\end{proof}

\begin{proof}[Proof of Theorem \ref{thm-hqe}]
Let $G=2\sup_{x\in D}|Du(x)|^2$ and $\phi(z)=(1-z/G)^{-1/8}.$ Define the function
$$\Phi(x;\eta)=(-u(x))^4\phi\left(\frac{|Du(x)|^2}{2}\right)D_{\eta\eta}u(x),\ x\in \o{D},|\eta|=1.$$
Suppose that
$\Phi(x_0;\eta_0)=\max\{\Phi(x;\eta)|x\in \o{D},|\eta|=1\}.$
Since $u=0$ on $\partial D$, then the point $x_0$ lies in $\o{D}\b\partial D.$ Without loss of generality, we may assume that $D^2u(x_0)=\mbox{diag}(u_{11}(x_0),u_{22}(x_0),\dots,u_{nn}(x_0))$ with $u_{11}(x_0)\geq u_{22}(x_0)\geq\dots\geq u_{nn}(x_0)>0$ and $\eta_0=e_1$, where $u_{ii}=\partial^2 u/\partial x_i^2$.
Then $$\Phi=\Phi(x;e_1)=(-u(x))^4\phi\left(\frac{|Du(x)|^2}{2}\right)u_{11}(x)$$
and its maximum is attained at $x_0$ and the eigenvalues of $D^2u(x_0)$ are $\lambda=(\lambda_1,\dots,\lambda_n)=(u_{11}(x_0),\dots,u_{nn}(x_0)).$

 Set
$$f(D^2u)=\left(\frac{S_k(D^2u)}{S_{l}(D^2u)}\right)^{\frac{1}{k-l}},$$
then $u$ satisfies
\begin{equation}\label{j-eq1}f(D^2u)=1\ \ \mbox{in}\ \ D.\end{equation}
By differentiating \eqref{j-eq1} with respect to $x_\alpha$, we have that
\begin{equation}\label{j-*}\sum_{i,j=1}^{n}f_{ij}u_{ij\alpha}=0,\end{equation}
where $f_{ij}=f_{ij}(D^2u)=\frac{\partial f(D^2u)}{\partial u_{ij}}.$
After differentiating \eqref{j-*} for $x_{\alpha}$ again, we can get that
\begin{align}\label{j-eq3}&\sum_{i,j=1}^{n}\sum_{r,s=1}^{n}f_{ij,rs}u_{ij\alpha}u_{rs\alpha}+\sum_{i,j=1}^{n}f_{ij}u_{ij\alpha\alpha}=0.\end{align}

Let $\mu=\mu(t)=t^{\frac{1}{k-l}}$, then
 $$f(D^2u)=\mu\left(\frac{S_k(D^2u)}{S_l(D^2u)}\right).$$
 The following calculations are all at $x_0$, unless otherwise stated.
So if $i=j$,
$$f_{ij}(D^2u)=f_{ii}(D^2u)=\mu^{\prime}\frac{S_{{k-1};i}S_l-S_kS_{l-1;i}}{S_l^2},$$
if $i\not=j$,
$$f_{ij}(D^2u)=\mu^{\prime}\frac{S_k^{ij}S_l-S_kS_l^{ij}}{S_l^2},$$
where $S_{k}^{ij}=\frac{\partial S_k(D^2u)}{\partial u_{ij}}.$

If $i=j,r=s,$
\begin{eqnarray*}&&f_{ij,rs}(D^2u)\\
&=&\mu^{\prime\prime}\dfrac{(S_{k-1;i}S_l-S_kS_{l-1;i})(S_{k-1;r}S_l-S_kS_{l-1;r})}{S_l^4}\\
&+&\mu^{\prime}\dfrac{(S_{k-2;ir}S_l+S_{k-1;i}S_{l-1;r}-S_{k-1;r}S_{l-1;i}-S_kS_{l-2;ir})S_l^2-2S_lS_{l-1;r}(S_{k-1;i}S_l-S_kS_{l-1;i})}{S_l^4},
\end{eqnarray*}
if $i\not=j,r=j,s=i,$
\begin{eqnarray*}&&f_{ij,rs}(D^2u)\\
&=&\mu^{\prime\prime}\dfrac{(S_k^{ij}S_l-S_kS_l^{ij})(S_k^{rs}S_l-S_kS_l^{rs})}{S_l^4}\\
&+&\mu^{\prime}\dfrac{(S_k^{ij,rs}S_l+S_k^{ij}S_l^{rs}-S_k^{rs}S_l^{ij}-S_kS_l^{ij,rs})S_l^2-2S_lS_l^{rs}(S_k^{ij}S_l-S_kS_l^{ij})}{S_l^4}\\
&=&\mu^{\prime}\dfrac{(-S_{k-2;ij}S_l+S_kS_{l-2;ij})S_l^2}{S_l^4},
\end{eqnarray*}
where the last equality uses the fact that $(D^2u)$ is diagonal, and $S_{k}^{ij,rs}=-S_{k-2;ij}$. Therefore,
$$f_{ij}(D^2u)=\\
\left\{\begin{array}{lll}
\mu^{\prime}\dfrac{S_{{k-1};i}S_l-S_kS_{l-1;i}}{S_l^2},\ &\mbox{if}\ i=j,\\
0,\ &\mbox{otherwise,}
\end{array}\right.
$$
and
$$f_{ij,rs}(D^2u)=\\
\left\{\begin{array}{lll}
\mu^{\prime\prime}\dfrac{(S_{k-1;i}S_l-S_kS_{l-1;i})(S_{k-1;r}S_l-S_kS_{l-1;r})}{S_l^4}\\
+\mu^{\prime}\dfrac{(S_{k-2;ir}S_l+S_{k-1;i}S_{l-1;r}-S_{k-1;r}S_{l-1;i}-S_kS_{l-2;ir})S_l^2}{S_l^4}\\
\vspace{5mm}
-\mu^{\prime}\dfrac{2S_lS_{l-1;r}(S_{k-1;i}S_l-S_kS_{l-1;i})}{S_l^4},\ &\mbox{if}\ i=j,r=s,\\
\vspace{5mm}
\mu^{\prime}\dfrac{(-S_{k-2;ij}S_l+S_kS_{l-2;ij})S_l^2}{S_l^4},\ &\mbox{if}\ i\not=j,r=j,s=i,\\
0,\ &\mbox{otherwise.}
\end{array}\right.
$$
As a result,
\begin{eqnarray*}
&&\sum_{i,j=1}^{n}\sum_{r,s=1}^{n}f_{ij,rs}(D^2u)u_{ij\alpha}u_{rs\alpha}\\
&=&\sum_{i=j,r=s}f_{ij,rs}(D^2u)u_{ij\alpha}u_{rs\alpha}+\sum_{i\not=j,r=j,s=i}f_{ij,rs}(D^2u)u_{ij\alpha}u_{rs\alpha}
\end{eqnarray*}
\begin{eqnarray*}
&=&\sum_{i=j,r=s}\left[\mu^{\prime\prime}\dfrac{(S_{k-1;i}S_l-S_kS_{l-1;i})(S_{k-1;r}S_l-S_kS_{l-1;r})}{S_l^4}\right.\\
&&+\mu^{\prime}\dfrac{(S_{k-2;ir}S_l+S_{k-1;i}S_{l-1;r}-S_{k-1;r}S_{l-1;i}-S_kS_{l-2;ir})S_l^2}{S_l^4}\\
&&\left.-\mu^{\prime}\dfrac{2S_lS_{l-1;r}(S_{k-1;i}S_l-S_kS_{l-1;i})}{S_l^4}\right]u_{ij\alpha}u_{rs\alpha}\\
&&+\sum_{i\not=j,r=j,s=i}\mu^{\prime}\dfrac{(-S_{k-2;ij}S_l+S_kS_{l-2;ij})S_l^2}{S_l^4}u_{ij\alpha}u_{rs\alpha}\\
&=&\sum_{i,r=1}^{n}\frac{\partial ^2}{\partial \lambda_i\partial \lambda_j}\left(\frac{S_k(\lambda)}{S_l(\lambda)}\right)^{\frac{1}{k-l}}u_{ii\alpha}u_{rr\alpha}+\sum_{i\not=j}\mu^{\prime}\dfrac{(-S_{k-2;ij}S_l+S_kS_{l-2;ij})S_l^2}{S_l^4}u_{ij\alpha}u_{ij\alpha}.
\end{eqnarray*}
Since $(\frac{S_k(\lambda)}{S_l(\lambda)})^{\frac{1}{k-l}}$ is concave in $\lambda=(\lambda_1,\dots,\lambda_n)$, then the first term in the above equality is nonpositive, so
\begin{eqnarray}
\sum_{i,j=1}^{n}\sum_{r,s=1}^{n}f_{ij,rs}(D^2u)u_{ij\alpha}u_{rs\alpha}
&\leq&\sum_{i\not=j}\mu^{\prime}\dfrac{(-S_{k-2;ij}S_l+S_kS_{l-2;ij})S_l^2}{S_l^4}u_{ij\alpha}u_{ij\alpha}\nonumber\\
&=&\frac{1}{k-l}\sum_{i,j=1}^{n}\left(\frac{S_k}{S_l}\right)^{\frac{1}{k-l}-1}\dfrac{-S_{k-2;ij}S_l+S_kS_{l-2;ij}}{S_l^2}u_{ij\alpha}^{2}.\label{j-**}
\end{eqnarray}

Since $f_{ij}=0, i\not=j$ at $x_0$, from \eqref{j-eq3} and \eqref{j-**}, we have that
\begin{eqnarray*}\sum_{i=1}^{n}f_{ii}u_{ii\alpha\alpha}&=&-\sum_{i,j=1}^{n}\sum_{r,s=1}^{n}f_{ii,rs}u_{ij\alpha}u_{rs\alpha}\\
&\geq&\frac{1}{k-l}\sum_{ij=1}^{n}\left(\frac{S_k}{S_l}\right)^{\frac{1}{k-l}-1}\frac{S_{k-2;ij}S_l-S_kS_{l-2;ij}}{S_l^2}u_{ij\alpha}^2.
\end{eqnarray*}
Let $\alpha=1$ and multiply $1/u_{11}$ on both sides in the above inequality, then
\begin{equation}\label{j-eq4}\sum_{i=1}^{n}f_{ii}\frac{u_{11 ii}}{u_{11}}
\geq\frac{1}{k-l}\sum_{i,j=1}^{n}\left(\frac{S_k}{S_l}\right)^{\frac{1}{k-l}-1}\frac{S_{k-2;ij}S_l-S_kS_{l-2;ij}}{S_l^2}\frac{u_{1 ij}^2}{u_{11}}.
\end{equation}
Let $\mathcal{L}$ denote the linearized operator of $f(D^2u)=1$ at $x_0$, then
$$\mathcal{L}=\sum_{i,j=1}^{n}f_{ij}(D^2u(x_0))D_{ij}.$$
Since the maximum of $\Phi$ is attained at $x_0$, by calculating, we can get at $x_0$ that
\begin{equation}\label{j-2}(\log\Phi)_i=\frac{4u_i}{u}+\frac{\phi_i}{\phi}+\frac{u_{11i}}{u_{11}}=0,\end{equation}
\begin{equation}\label{j-3}(\log\Phi)_{ii}=4\left(\frac{u_{ii}}{u}-\frac{u_i^2}{u^2}\right)+\frac{\phi_{ii}}{\phi}-\frac{\phi_i^2}{\phi^2}+\frac{u_{11ii}}{u_{11}}-\frac{u_{11i}^2}{u_{11}^2}\leq 0,\end{equation}
\begin{equation}\label{j-5}\phi_i=\phi^{\prime}\left(\frac{|Du|^2}{2}\right)u_iu_{ii},\end{equation}
\begin{equation}\label{j-2.7-1}\phi_{ii}=\phi^{\prime\prime}\left(\frac{|Du|^2}{2}\right)u_i^2u_{ii}^2+\phi^{\prime}\left(\frac{|Du|^2}{2}\right)\left(u_{ii}^2+\sum_{j=1}^nu_ju_{iij}\right),\end{equation}
where $u_i=\partial u/\partial x_i,i=1,\dots,n.$
So
\begin{eqnarray*}\mathcal{L}(\log\Phi)&=&\sum_{i,j=1}^{n}f_{ij}(D^2u)D_{ij}(\log\Phi)\\
&=&\sum_{i=1}^{n}f_{ii}(D^2u)\left[4\left(\frac{u_{ii}}{u}-\frac{u_i^2}{u^2}\right)+\frac{\phi_{ii}}{\phi}-\frac{\phi_i^2}{\phi^2}+\frac{u_{11ii}}{u_{11}}-\frac{u_{11i}^2}{u_{11}^2}\right]\\
&\leq &0,
\end{eqnarray*}
where we used the fact that $f_{ij}=0, i\not=j$ at $x_0$.
Substituting \eqref{j-5} and \eqref{j-2.7-1} into the above inequality, we have
\begin{equation*}
\sum_{i=1}^{n}f_{ii}\left[4\left(\frac{u_{ii}}{u}-\frac{u_i^2}{u^2}\right)+\frac{\phi^{\prime\prime}u_i^2u_{ii}^2}{\phi}+\frac{\phi^{\prime}u_{ii}^2}{\phi}+\frac{\phi^{\prime}\sum_{j=1}^{n}u_ju_{iij}}{\phi}
-\frac{\phi^{\prime 2}u_i^2u_{ii}^2}{\phi^2}+\frac{u_{11ii}}{u_{11}}-\frac{u_{11i}^2}{u_{11}^2}\right]\leq 0,
\end{equation*}
i.e.,
\begin{eqnarray}
&&\sum_{i=1}^{n}f_{ii}\frac{\phi^{\prime}\sum_{j=1}^{n}u_ju_{iij}}{\phi}
+\sum_{i=1}^{n}f_{ii}\frac{u_{11ii}}{u_{11}}\nonumber\\
+&&\sum_{i=1}^{n}f_{ii}\left[4\left(\frac{u_{ii}}{u}-\frac{u_i^2}{u^2}\right)+\frac{\phi^{\prime\prime}u_i^2u_{ii}^2}{\phi}+\frac{\phi^{\prime}u_{ii}^2}{\phi}-\frac{\phi^{\prime 2}u_i^2u_{ii}^2}{\phi^2}-\frac{u_{11i}^2}{u_{11}^2}\right]\leq 0.\label{j-2.14-1}
\end{eqnarray}
Similar to \eqref{j-*}, we know that for $j=1,\dots,n$,
$$\sum_{i=1}^{n}f_{ii}u_{iij}=0.$$
So
\begin{equation*}\sum_{i=1}^{n}f_{ii}\frac{\phi^{\prime}\sum_{j=1}^{n}u_ju_{iij}}{\phi}
=\frac{\phi^{\prime}}{\phi}\sum_{i,j=1}^{n}u_j\left[\frac{u_{jt}}{(k-l)u_t}+(-u_t)^{\frac{1}{k-l}}f_{ii}u_{iij}\right]=0.
\end{equation*}
Consequently, by \eqref{j-2.14-1} and \eqref{j-eq4},
\begin{eqnarray}
&&\sum_{i=1}^{n}f_{ii}\left[4\left(\frac{u_{ii}}{u}-\frac{u_i^2}{u^2}\right)+\frac{\phi^{\prime\prime}u_i^2u_{ii}^2}{\phi}+\frac{\phi^{\prime}u_{ii}^2}{\phi}-\frac{\phi^{\prime 2}u_i^2u_{ii}^2}{\phi^2}-\frac{u_{11i}^2}{u_{11}^2}\right]\nonumber\\
&&+\frac{1}{k-l}\sum_{i,j=1}^n\left(\frac{S_k}{S_l}\right)^{\frac{1}{k-l}-1}\frac{S_{k-2;ij}S_l-S_kS_{l-2;ij}}{S_l^2}\frac{u_{1ij}^2}{u_{11}}
\label{j-6}\leq 0.
\end{eqnarray}

From now on, we discuss the case of $k=n$. Substituting
$$\frac{u_{111}}{u_{11}}=-\left(\frac{\phi_1}{\phi}+\frac{4u_1}{u}\right),\frac{u_i}{u}=-\frac{1}{4}\left(\frac{\phi_i}{\phi}+\frac{u_{11i}}{u_{11}}\right),i=1,\dots,n$$
into \eqref{j-6}, we have
\begin{eqnarray}
0&\geq & f_{11}\left[4\left(\frac{u_{11}}{u}-\frac{u_1^2}{u^2}\right)+\frac{\phi^{\prime\prime}}{\phi}u_1^2u_{11}^2+\frac{\phi^{\prime}}{\phi}u_{11}^2-\frac{\phi^{\prime 2}}{\phi^2}u_1^2u_{11}^2-\left(\frac{\phi_1}{\phi}+\frac{4u_1}{u}\right)^2\right]\nonumber\\
&&+\sum_{i=2}^nf_{ii}\left[\frac{4u_{ii}}{u}-\frac{1}{4}\left(\frac{\phi_i}{\phi}+\frac{u_{11i}}{u_{11}}\right)^2+\frac{\phi^{\prime\prime}}{\phi}u_i^2u_{ii}^2+\frac{\phi^\prime}{\phi}u_{ii}^2
-\frac{\phi^{\prime 2}}{\phi^2}u_i^2u_{ii}^2-\frac{u_{11i}^2}{u_{11}^2}\right]\nonumber\\
&&+\frac{1}{n-l}\sum_{i,j=1}^n\left(\frac{S_n}{S_l}\right)^{\frac{1}{n-l}-1}\frac{S_{n-2;ij}S_l-S_nS_{l-2;ij}}{S_l^2}\frac{u_{1ij}^2}{u_{11}}\nonumber\\
&\geq &\left\{\sum_{i=1}^nf_{ii}\left[\frac{4u_{ii}}{u}+\left(\frac{\phi^{\prime\prime}}{\phi}-\frac{3\phi^{\prime 2}}{\phi^2}\right)u_i^2u_{ii}^2+\frac{\phi^{\prime}}{\phi}u_{ii}^2\right]
-36f_{11}\frac{u_1^2}{u^2}\right\}\nonumber\\
&&+\left[-\frac{3}{2}\sum_{i=2}^nf_{ii}\frac{u_{11i}^2}{u_{11}^2}
+\frac{1}{n-l}\sum_{i,j=1}^n\left(\frac{S_n}{S_l}\right)^{\frac{1}{n-l}-1}\frac{S_{n-2;ij}S_l-S_nS_{l-2;ij}}{S_l^2}\frac{u_{1ij}^2}{u_{11}}\right]\nonumber\\
\label{j-16-1}&=:&I_1+I_2
\end{eqnarray}
 where
$$I_1=\sum_{i=1}^nf_{ii}\left[\frac{4u_{ii}}{u}+\left(\frac{\phi^{\prime\prime}}{\phi}-\frac{3\phi^{\prime 2}}{\phi^2}\right)u_i^2u_{ii}^2+\frac{\phi^{\prime}}{\phi}u_{ii}^2\right]
-36f_{11}\frac{u_1^2}{u^2},$$
$$I_2=-\frac{3}{2}\sum_{i=2}^nf_{ii}\frac{u_{11i}^2}{u_{11}^2}
+\frac{1}{n-l}\sum_{i,j=1}^n\left(\frac{S_n}{S_l}\right)^{\frac{1}{n-l}-1}\frac{S_{n-2;ij}S_l-S_nS_{l-2;ij}}{S_l^2}\frac{u_{1ij}^2}{u_{11}},$$
and we used \eqref{j-5} and the inequalities
$$\left(\frac{\phi_1}{\phi}+\frac{4u_1}{u}\right)^2\leq 2\frac{\phi_1^2}{\phi^2}+\frac{32u_1^2}{u^2},$$
and for $i=1,\dots,n$,
\begin{eqnarray*}\frac{1}{4}\left(\frac{\phi_i}{\phi}+\frac{u_{11i}}{u_{11}}\right)^2&\leq& \frac{1}{2}\frac{\phi_i^2}{\phi^2}+\frac{1}{2}\frac{u_{11i}^2}{u_{11}^2}\\
&\leq &2\frac{\phi_i^2}{\phi^2}+\frac{1}{2}\frac{u_{11i}^2}{u_{11}^2}.
\end{eqnarray*}

Due to \eqref{ele-1}, then
\begin{eqnarray}
\sum_{i=1}^{n}f_{ii}u_{ii}&=&\sum_{i=1}^{n}\mu^{\prime}\frac{S_{n-1;i}S_l-S_nS_{l-1;i}}{S_l^2}\lambda_i\nonumber\\
&=&\frac{\mu^{\prime}}{S_l^2}\left(\sum_{i=1}^{n}S_l S_{n-1;i}\lambda_i-\sum_{i=1}^{n}S_nS_{l-1;i}\lambda_i\right)\nonumber
\end{eqnarray}
\begin{eqnarray}
&=&\frac{\mu^{\prime}}{S_l^2}(nS_nS_l-lS_nS_l)\nonumber\\
\label{j-e2}&=&(n-l)\mu^{\prime}\frac{S_n}{S_l}=\left(\frac{S_n}{S_l}\right)^{\frac{1}{n-l}}=f(D^2u)=1.
\end{eqnarray}
By \eqref{2.16-1}, we have
\begin{equation}\label{j-2.18-1}
f_{ii}=\mu^{\prime}\frac{S_{n-1;i}S_l-S_nS_{l-1;i}}{S_l^2}=\mu^{\prime}\frac{S_{n-1;i}S_{l;i}}{S_l^2}.
\end{equation}
By \eqref{j-e2} and $\frac{\phi^{\prime\prime}}{\phi}-\frac{3\phi^{\prime 2}}{\phi^2}\geq 0$ and $\phi^{\prime}\geq 0,$
then $I_1$ can be estimated
\begin{eqnarray}
I_1&\geq& \frac{4}{u}+\theta_8f_{11}u_{11}^2-C_2\frac{f_{11}}{u^2},
\end{eqnarray}
where $\theta_8=\theta_8(\phi)$ and $C_2=C_2(||u||_{C^1(D)})$. We can assume that $\Phi(x_0)$ is large such that
\begin{equation}\label{j-16-3}u(x_0)^2u_{11}(x_0)^2\geq \frac{2C_2}{\theta_8},\end{equation}
as otherwise, $\Phi(x_0)\leq C,$ then \eqref{sec-est} is obvious.
So in the following, we always assume that \eqref{j-16-3} holds. Then
\begin{eqnarray}
I_1&\geq & \frac{1}{2}\theta_8f_{11}u_{11}^2+\frac{4}{u}.\label{j-16-2}
\end{eqnarray}

 $I_2$ can be estimated by \eqref{j-2.18-1},
\begin{eqnarray}
I_2&\geq &-\frac{3}{2}\frac{1}{n-l}\left(\frac{S_n}{S_l}\right)^{\frac{1}{n-l}-1}\sum_{i=2}^n\frac{S_{n-1;i}S_l-S_nS_{l-1;i}}{S_l^2}\frac{u_{11i}^2}{u_{11}^2}\nonumber\\
&&+2\frac{1}{n-l}\left(\frac{S_n}{S_l}\right)^{\frac{1}{n-l}-1}\sum_{i=2}^n\frac{S_{n-2;1i}S_l-S_nS_{l-2;1i}}{S_l^2}\frac{u_{11i}^2}{u_{11}}\nonumber\\
&=&2\frac{1}{n-l}\left(\frac{S_n}{S_l}\right)^{\frac{1}{n-l}-1}\nonumber\\
&&\sum_{i=2}^n\left(\frac{S_{n-2;1i}S_l-S_nS_{l-2;1i}}{S_l^2}-\frac{3}{4}\frac{S_{n-1;i}S_l-S_nS_{l-1;i}}{\lambda_1S_l^2}\right)\frac{u_{11i}^2}{\lambda_1}.\label{j-2.27*}
\end{eqnarray}
Due to \eqref{2.15-1}, we have
$$S_{l-1;i}=S_{l-2;1i}\lambda_1+S_{l-1;1i},$$
and
$$S_{n-1;i}=S_{n-2;1i}\lambda_1+S_{n-1;1i}.$$
So
\begin{equation}\label{j-2.27-2} S_{l-2;1i}\lambda_1=S_{l-1;i}-S_{l-1;1i},\end{equation}
and by $S_{n-1;1i}=0$, we have
\begin{equation}\label{j-2.27-1}S_{n-2;1i}\lambda_1=S_{n-1;i}.\end{equation}
Then by \eqref{j-2.27-1}, \eqref{j-2.27-2} and \eqref{2.16-1},
\begin{eqnarray*}
&&\lambda_1(S_{n-2;1i}S_l-S_nS_{l-2;1i})\\
&=&\lambda_1S_{n-2;1i}S_l-S_nS_{l-2;1i}\lambda_1\\
&=&S_lS_{n-1;i}-S_n(S_{l-1;i}-S_{l-1;1i})\\
&=&S_lS_{n-1;i}-S_nS_{l-1;i}+S_nS_{l-1;1i}\\
&=&S_{n-1;i}S_{l;i}+S_nS_{l-1;1i}.
\end{eqnarray*}
Therefore, by \eqref{2.16-1},
\begin{eqnarray*}
&&4(S_{n-1;i}S_{l;i}+S_nS_{l-1;1i})-3(S_{n-1;i}S_{l}-S_nS_{l-1;i})\\
&=&4(S_{n-1;i}S_{l;i}+S_nS_{l-1;1i})-3S_{n-1;i}S_{l;i}\\
&=&S_{n-1;i}S_{l;i}+4S_nS_{l-1;1i}>0.
\end{eqnarray*}
Thus
\begin{equation}\label{ell-1}\frac{S_{n-2;1i}S_l-S_nS_{l-2;1i}}{S_l^2}-\frac{3}{4}\frac{S_{n-1;i}S_l-S_nS_{l-1;i}}{\lambda_1S_l^2}>0.\end{equation}
Then by \eqref{j-2.27*}
\begin{equation}\label{j-18-1}I_2\geq 0.\end{equation}
By \eqref{j-16-2} and \eqref{j-18-1}, then \eqref{j-16-1} can be converted to
\begin{equation}\label{j-2.28-1} 0\geq \frac{1}{2}\theta_8f_{11}u_{11}^2+\frac{4}{u}.\end{equation}
Multiplying $(-u)^4\phi$ on both sides in \eqref{j-2.28-1}, we have
\begin{equation}\label{j-18-2}
0\geq \frac{1}{2}\theta_8f_{11}u_{11}^2(-u)^4\phi-4(-u)^3\phi.
\end{equation}

According to \cite{chen}, we can know that
\begin{equation}\label{j-f11-1}f_{11}\geq \sum_{i=1}^nf_{ii},
\end{equation}
for $k=n$.
By the proof of Lemma 3.10 in \cite{chen}, we know that
$$\sum_{i=1}^n f_{ii}\geq \mu^{\prime}\frac{k-l}{k}(n-k+1)\frac{S_{k-1}}{S_l}.$$
So for $k=n$, then
\begin{equation}\label{j-f11-2}\sum_{i=1}^n f_{ii}\geq \mu^{\prime}\frac{n-l}{n}\frac{S_{n-1}}{S_l}.\end{equation}
So by \eqref{j-f11-1} and \eqref{j-f11-2},
\begin{equation}\label{j-f11-3}f_{11}u_{11}^2\geq \sum_{i=1}^n f_{ii}u_{11}^2\geq \mu^{\prime}\frac{n-l}{n}\frac{S_{n-1}}{S_l}u_{11}^2.\end{equation}


Since $S_{n-1}=S_{n-2;1}\lambda_1+S_{n-1;1}$, then
\begin{eqnarray*}
\lambda_1S_{n-1}&=&\lambda_1^2S_{n-2;1}+\lambda_1 S_{n-1;1}\\
&=&\lambda_1^2S_{n-2;1}+S_{n}\\
&\geq& S_n.
\end{eqnarray*}
Due to \eqref{j-f11-3},
\begin{eqnarray}\label{j-f11-4}f_{11}u_{11}^2&\geq& \mu^{\prime}\frac{n-l}{n}\frac{\lambda_1S_{n-1}}{S_l}u_{11}\nonumber\\
&\geq &\mu^{\prime}\frac{n-l}{n}\frac{S_{n}}{S_l}u_{11}\nonumber\\
&=&\frac{1}{n}u_{11}.
\end{eqnarray}

Consequently, by \eqref{j-18-2}, we have that
$$0\geq \frac{\theta_8}{2n}(-u)^4\phi u_{11}-4(-u)^3\phi.$$
Thus
\begin{eqnarray*}
\Phi=(-u)^4\phi u_{11}&\leq& \frac{8n(-u)^3\phi}{\theta_8}\\
&\leq & C_0(n,||u||_{C^1(D)}).
\end{eqnarray*}
Then $(-u)^4|D^2u|$ can be estimated.
\end{proof}

\section{Gradient estimate and Pogorelov type results}

In this section, we give the proof of the interior gradient estimate and Pogorelov type estimate for parabolic Hessian quotient equation.

\begin{proof}[Proof of Theorem \ref{gradient}]
For $(x,t)\in \o{\Omega}\subset B_r\times(-\infty,0]$ and $|\xi|=1$, let
$$\Psi(x,t;\xi)=\varphi(u(x,t))\omega(x)u_{\xi}(x,t),$$
where $\omega(x)=1-|x|^2/r^2,\varphi(\tau)=(M-\tau)^{-\frac{1}{2}}$ and $M=3\mbox{osc}_\Omega u+\sup_{\Omega}u.$  Suppose that
$$\Psi(x_0,t_0;\xi_0)=\max\{\Psi(x,t;\xi)|(x,t)\in \o{\Omega},|\xi|=1\}. $$
 Obviously, $(x_0,t_0)\in \o{\Omega}\backslash \partial_p\Omega.$ By rotating the coordinates $(x_1,\dots,x_n)$, we may assume that $\xi_0=e_1=(1,0,\dots,0)$. Clearly, $u_i(x_0,t_0)=0,i=2,\dots,n.$ Then
$$\Psi(x,t)=\varphi(u(x,t))\omega(x)u_{1}(x,t)$$
attains its maximum at $(x_0,t_0)$ and
$$\psi(x,t)=\log\Psi(x,t)=\log\varphi+\log\omega+\log u_{1}$$
attains its maximum at $(x_0,t_0)$. By calculation, we have for $i,j=1,2,\dots,n$,
\begin{equation}\label{1-1}
0=\psi_i=(\log \Psi)_i=\frac{\varphi_i}{\varphi}+\frac{\omega_i}{\omega}+\frac{u_{1i}}{u_1}.
\end{equation}
\begin{equation}\label{psiij}
\psi_{ij}=\frac{\varphi_{ij}}{\varphi}-\frac{\varphi_i\varphi_j}{\varphi^2}+\frac{\omega_{ij}}{\omega}-\frac{\omega_i\omega_j}{\omega^2}+\frac{u_{1ij}}{u_1}-\frac{u_{1i}u_{1j}}{u_1^2}.
\end{equation}
\begin{equation}\label{1-2}
0\leq \psi_t=(\log \Psi)_t=\frac{\varphi^{\prime}}{\varphi}u_t+\frac{u_{1t}}{u_1}.
\end{equation}
Set
\begin{equation}\label{3-1}F(D^2u)=\frac{S_k(D^2u)}{S_l(D^2u)}.\end{equation}
Then
\begin{equation}\label{F-1}-u_tF(D^2u)=1.\end{equation}
Let $L$ be the linearized operator of \eqref{3-1} at $(x_0,t_0)$. Then
$$L=\frac{1}{u_t(x_0,t_0)}D_t-u_t(x_0,t_0)\sum_{i,j=1}^{n}F_{ij}(D^2u(x_0,t_0))D_{ij},$$
where $$F_{ij}=\frac{\partial F}{\partial u_{ij}}(D^2u).$$
Then at $(x_0,t_0)$, according to \eqref{1-1} and \eqref{1-2}, we have that
\begin{align}\label{F-4}
0&\geq L\psi\nonumber\\
&=\frac{1}{u_t}D_t\psi-u_t\sum_{i,j=1}^{n}F_{ij}D_{ij}\psi\nonumber\\
&=\frac{1}{u_t}\left(\frac{\varphi^{\prime}}{\varphi}u_t+\frac{u_{1t}}{u_1}\right)
-u_t\sum_{i,j=1}^{n}F_{ij}\left(\frac{u_{1ij}}{u_1}-\frac{u_{1i}u_{1j}}{u_1^2}+\frac{\varphi_{ij}}{\varphi}-\frac{\varphi_i\varphi_j}{\varphi^2}+\frac{\omega_{ij}}{\omega}-\frac{\omega_i\omega_j}{\omega^2}\right)\nonumber\\
&=\frac{1}{u_1}\frac{u_{1t}}{u_t}-\frac{1}{u_1}u_t\sum_{i,j=1}^{n}F_{ij}u_{1ij}+\frac{\varphi^{\prime}}{\varphi}+u_t\sum_{i,j=1}^{n}F_{ij}\left(\frac{u_{1i}u_{1j}}{u_1^2}-\frac{\varphi_{ij}}{\varphi}+\frac{\varphi_i\varphi_j}{\varphi^2}-\frac{\omega_{ij}}{\omega}+\frac{\omega_i\omega_j}{\omega^2}\right).
\end{align}
From \eqref{F-1}, we have
\begin{equation}\label{F-2}
F(D^2u)=-\frac{1}{u_t}.
\end{equation}
By differentiating \eqref{F-1} on both sides with respect to $x_1$, we have
$$-u_{1t}F(D^2u)-u_t\sum_{i,j=1}^{n}F_{ij}(D^2u)u_{1ij}=0.$$
Due to \eqref{F-2}, we have
\begin{equation}\label{F-3}
\frac{u_{1t}}{u_t}-u_t\sum_{i,j=1}^{n}F_{ij}u_{1ij}=0.
\end{equation}
As $\varphi^{\prime}/\varphi\geq 0$, by \eqref{F-4} and \eqref{F-3}, we know
\begin{align*}
u_t\sum_{i,j=1}^{n}F_{ij}\left(\frac{u_{1i}u_{1j}}{u_1^2}-\frac{\varphi_{ij}}{\varphi}+\frac{\varphi_i\varphi_j}{\varphi^2}-\frac{\omega_{ij}}{\omega}+\frac{\omega_i\omega_j}{\omega^2}\right)\leq 0 .
\end{align*}
Since $-u_t\geq m_1$, the remainders of the proof is the same as that of \cite{chen}. Here we omit the proof.

\end{proof}

\begin{proof}[Proof of Theorem \ref{thm-ns}]
Let $G=2\sup_{(x,t)\in \Omega}|Du(x,t)|^2$ and $\phi(z)=(1-z/G)^{-1/8}.$ Define the function
$$\Phi(x,t;\eta)=(-u(x,t))^4\phi\left(\frac{|Du(x,t)|^2}{2}\right)D_{\eta\eta}u(x,t),\ (x,t)\in \o{\Omega},|\eta|=1.$$
Suppose that
$\Phi(x_0,t_0;\eta_0)=\max\{\Phi(x,t;\eta)|(x,t)\in \o{\Omega},|\eta|=1\}.$
Since $u=0$ on $\partial_p\Omega$, then the point $(x_0,t_0)$ lies in $\o{\Omega}\b\partial_p\Omega.$ Without loss of generality, we may assume that $D^2u(x_0,t_0)=\mbox{diag}(u_{11}(x_0,t_0),u_{22}(x_0,t_0),\dots,u_{nn}(x_0,t_0))$ with $u_{11}(x_0,t_0)\geq u_{22}(x_0,t_0)\geq\dots\geq u_{nn}(x_0,t_0)>0$ and $\eta_0=e_1$, where $u_{ii}=\partial^2 u/\partial x_i^2$.
Then $$\Phi=\Phi(x,t;e_1)=(-u(x,t))^4\phi\left(\frac{|Du(x,t)|^2}{2}\right)u_{11}(x,t)$$
and its maximum is attained at $(x_0,t_0)$ and the eigenvalues of $D^2u(x_0,t_0)$ are $\lambda=(\lambda_1,\dots,\lambda_n)=(u_{11}(x_0,t_0),\dots,u_{nn}(x_0,t_0)).$ It is adequate to consider $\lambda_1=u_{11}(x_0,t_0)\geq 1.$

 Set
$$f(D^2u)=\left(\frac{S_k(D^2u)}{S_{l}(D^2u)}\right)^{\frac{1}{k-l}},$$
then $u$ satisfies
\begin{equation}\label{eq1}(-u_t)^{\frac{1}{k-l}}f(D^2u)=1\ \ \mbox{in}\ \ \Omega.\end{equation}
By differentiating \eqref{eq1} with respect to $x_\alpha$, we have that
\begin{equation}\label{*}\frac{1}{k-l}(-u_t)^{\frac{1}{k-l}-1}(-u_{\alpha t})f(D^2u)+(-u_t)^{\frac{1}{k-l}}\sum_{i,j=1}^{n}f_{ij}u_{ij\alpha}=0,\end{equation}
where $f_{ij}=f_{ij}(D^2u)=\frac{\partial f(D^2u)}{\partial u_{ij}}.$
Utilizing \eqref{eq1}, we can get that $f(D^2u)=(-u_t)^{-\frac{1}{k-l}}$, so substitute it into \eqref{*}, then
\begin{equation}\label{1}-\frac{1}{k-l}(-u_t)^{-1}u_{\alpha t}+(-u_t)^{\frac{1}{k-l}}\sum_{i,j=1}^{n}f_{ij}u_{ij\alpha}=0.\end{equation}
Multiplying $(-u_t)^{-\frac{1}{k-l}}$ on both sides of \eqref{1}, we have
\begin{equation}\label{eq2}-\frac{1}{k-l}(-u_t)^{-(\frac{1}{k-l}+1)}u_{\alpha t}+\sum_{i,j=1}^{n}f_{ij}u_{ij\alpha}=0.\end{equation}
After differentiating \eqref{eq2} for $x_{\alpha}$ again, we can get that
\begin{align*}&\frac{1}{k-l}\left(-\frac{1}{k-l}-1\right)(-u_t)^{-\frac{1}{k-l}-2}u_{\alpha t}^2-\frac{1}{k-l}(-u_t)^{-\frac{1}{k-l}-1}u_{\alpha\alpha t}\\
&+\sum_{i,j=1}^{n}\sum_{r,s=1}^{n}f_{ij,rs}u_{ij\alpha}u_{rs\alpha}+\sum_{i,j=1}^{n}f_{ij}u_{ij\alpha\alpha}=0.\end{align*}
Multiplying $(-u_t)^{\frac{1}{k-l}}$ on both sides of the above equality,
\begin{eqnarray*}
&-\dfrac{1}{k-l}\left(\dfrac{1}{k-l}+1\right)(-u_t)^{-2}u_{\alpha t}^2-\dfrac{1}{k-l}(-u_t)^{-1}u_{\alpha\alpha t}\\
&+(-u_t)^{\frac{1}{k-l}}\sum_{i,j=1}^{n}\sum_{r,s=1}^{n}f_{ij,rs}u_{ij\alpha}u_{rs\alpha}+(-u_t)^{\frac{1}{k-l}}\sum_{i,j=1}^{n}f_{ij}u_{ij\alpha\alpha}=0,
\end{eqnarray*}
i.e.,
\begin{eqnarray}
&-\left(\dfrac{1}{k-l}+1\right)\dfrac{u_{\alpha t}^2}{(k-l)u_t^{2}}+\dfrac{u_{\alpha\alpha t}}{(k-l)u_t}\nonumber\\
&+(-u_t)^{\frac{1}{k-l}}\sum_{i,j=1}^{n}\sum_{r,s=1}^{n}f_{ij,rs}u_{ij\alpha}u_{rs\alpha}+(-u_t)^{\frac{1}{k-l}}\sum_{i,j=1}^{n}f_{ij}u_{ij\alpha\alpha}=0\label{eq3}.
\end{eqnarray}

Let $\mu=\mu(t)=t^{\frac{1}{k-l}}$, then
 $$f(D^2u)=\mu\left(\frac{S_k(D^2u)}{S_l(D^2u)}\right).$$
 The following calculations are all at $(x_0,t_0)$, unless otherwise stated. Since $f_{ij}=0, i\not=j$ at $(x_0,t_0)$, from \eqref{eq3} and \eqref{j-**}, we have that
\begin{eqnarray*}\frac{u_{\alpha\alpha t}}{(k-l)u_t}+(-u_t)^{\frac{1}{k-l}}\sum_{i=1}^{n}f_{ii}u_{ii\alpha\alpha}&=&\left(\frac{1}{k-l}+1\right)\frac{u_{\alpha t}^2}{(k-l)u_t^2}-(-u_t)^{\frac{1}{k-l}}\sum_{i,j=1}^{n}\sum_{r,s=1}^{n}f_{ii,rs}u_{ij\alpha}u_{rs\alpha}\\
&\geq&(-u_t)^{\frac{1}{k-l}}\frac{1}{k-l}\sum_{i,j=1}^{n}\left(\frac{S_k}{S_l}\right)^{\frac{1}{k-l}-1}\frac{S_{k-2;ij}S_l-S_kS_{l-2;ij}}{S_l^2}u_{ij\alpha}^2.
\end{eqnarray*}
Let $\alpha=1$ and multiply $1/u_{11}$ on both sides in the above inequality, then
\begin{eqnarray}&&\frac{u_{11 t}}{(k-l)u_tu_{11}}+(-u_t)^{\frac{1}{k-l}}\sum_{i=1}^{n}f_{ii}\frac{u_{11 ii}}{u_{11}}\nonumber\\
\label{eq4}&\geq&(-u_t)^{\frac{1}{k-l}}\frac{1}{k-l}\sum_{i,j=1}^{n}\left(\frac{S_k}{S_l}\right)^{\frac{1}{k-l}-1}\frac{S_{k-2;ij}S_l-S_kS_{l-2;ij}}{S_l^2}\frac{u_{1 ij}^2}{u_{11}}.
\end{eqnarray}
Let $\mathcal{L}$ denote the linearized operator of $(-u_t)^{\frac{1}{k-l}}f(D^2u)=1$ at $(x_0,t_0)$, then by \eqref{1},
$$\mathcal{L}=\frac{1}{(k-l)u_t(x_0,t_0)}D_t+(-u_t(x_0,t_0))^{\frac{1}{k-l}}\sum_{i,j=1}^{n}f_{ij}(D^2u(x_0,t_0))D_{ij}.$$
Since the maximum of $\Phi$ is attained at $(x_0,t_0)$, by calculating, we can get at $(x_0,t_0)$ that
\begin{equation}\label{2}(\log\Phi)_i=\frac{4u_i}{u}+\frac{\phi_i}{\phi}+\frac{u_{11i}}{u_{11}}=0,\end{equation}
\begin{equation}\label{3}(\log\Phi)_{ii}=4\left(\frac{u_{ii}}{u}-\frac{u_i^2}{u^2}\right)+\frac{\phi_{ii}}{\phi}-\frac{\phi_i^2}{\phi^2}+\frac{u_{11ii}}{u_{11}}-\frac{u_{11i}^2}{u_{11}^2}\leq 0,\end{equation}
\begin{equation}\label{4}(\log\Phi)_t=\frac{4u_t}{u}+\frac{\phi_t}{\phi}+\frac{u_{11t}}{u_{11}}\geq 0,\end{equation}
\begin{equation}\label{5}\phi_i=\phi^{\prime}\left(\frac{|Du|^2}{2}\right)u_iu_{ii},\end{equation}
\begin{equation}\label{2.7-1}\phi_{ii}=\phi^{\prime\prime}\left(\frac{|Du|^2}{2}\right)u_i^2u_{ii}^2+\phi^{\prime}\left(\frac{|Du|^2}{2}\right)\left(u_{ii}^2+\sum_{j=1}^nu_ju_{iij}\right),\end{equation}
\begin{equation}\label{2.7-2}\phi_{t}=\phi^{\prime}\left(\frac{|Du|^2}{2}\right)\sum_{j=1}^nu_ju_{jt},\end{equation}
where $u_i=\partial u/\partial x_i,i=1,\dots,n.$
So
\begin{eqnarray*}\mathcal{L}(\log\Phi)&=&\frac{1}{(k-l)u_t}D_t(\log\Phi)+(-u_t)^{\frac{1}{k-l}}\sum_{i,j=1}^{n}f_{ij}(D^2u)D_{ij}(\log\Phi)\\
&=&\frac{1}{(k-l)u_t}\left(\frac{4u_t}{u}+\frac{\phi_t}{\phi}+\frac{u_{11t}}{u_{11}}\right)\\
&&+(-u_t)^{\frac{1}{k-l}}\sum_{i=1}^{n}f_{ii}(D^2u)\left[4\left(\frac{u_{ii}}{u}-\frac{u_i^2}{u^2}\right)+\frac{\phi_{ii}}{\phi}-\frac{\phi_i^2}{\phi^2}+\frac{u_{11ii}}{u_{11}}-\frac{u_{11i}^2}{u_{11}^2}\right]\\
&\leq &0\ \ \mbox{at}\ \ (x_0,t_0),
\end{eqnarray*}
where we used the fact that $f_{ij}=0, i\not=j$ at $(x_0,t_0)$.
Substituting \eqref{5}, \eqref{2.7-1} and \eqref{2.7-2} into the above inequality, we have
\begin{eqnarray*}
&&\frac{4}{(k-l)u}-\frac{1}{k-l}(-u_t)^{-1}\frac{\phi^{\prime}\sum_{j=1}^nu_ju_{jt}}{\phi}+\frac{u_{11t}}{(k-l)u_tu_{11}}\\
+&&(-u_t)^{\frac{1}{k-l}}\sum_{i=1}^{n}f_{ii}\left[4\left(\frac{u_{ii}}{u}-\frac{u_i^2}{u^2}\right)+\frac{\phi^{\prime\prime}u_i^2u_{ii}^2}{\phi}+\frac{\phi^{\prime}u_{ii}^2}{\phi}+\frac{\phi^{\prime}\sum_{j=1}^{n}u_ju_{iij}}{\phi}\right.\\
&&\left.-\frac{\phi^{\prime 2}u_i^2u_{ii}^2}{\phi^2}+\frac{u_{11ii}}{u_{11}}-\frac{u_{11i}^2}{u_{11}^2}\right]\leq 0,
\end{eqnarray*}
i.e.,
\begin{eqnarray}
&&\frac{4}{(k-l)u}-\frac{1}{k-l}(-u_t)^{-1}\frac{\phi^{\prime}\sum_{j=1}^nu_ju_{jt}}{\phi}+(-u_t)^{\frac{1}{k-l}}\sum_{i=1}^{n}f_{ii}\frac{\phi^{\prime}\sum_{j=1}^{n}u_ju_{iij}}{\phi}\nonumber\\
+&&\frac{u_{11t}}{(k-l)u_tu_{11}}+(-u_t)^{\frac{1}{k-l}}\sum_{i=1}^{n}f_{ii}\frac{u_{11ii}}{u_{11}}\nonumber\\
+&&(-u_t)^{\frac{1}{k-l}}\sum_{i=1}^{n}f_{ii}\left[4\left(\frac{u_{ii}}{u}-\frac{u_i^2}{u^2}\right)+\frac{\phi^{\prime\prime}u_i^2u_{ii}^2}{\phi}+\frac{\phi^{\prime}u_{ii}^2}{\phi}-\frac{\phi^{\prime 2}u_i^2u_{ii}^2}{\phi^2}-\frac{u_{11i}^2}{u_{11}^2}\right]\leq 0.\label{2.14-1}
\end{eqnarray}
Similar to \eqref{1}, we know that
$$-\frac{1}{k-l}(-u_t)^{-1}u_{j t}+(-u_t)^{\frac{1}{k-l}}\sum_{i=1}^{n}f_{ii}u_{iij}=0\ \mbox{at}\ (x_0,t_0).$$
So
\begin{eqnarray*}&&-\frac{1}{k-l}(-u_t)^{-1}\frac{\phi^{\prime}\sum_{j=1}^nu_ju_{jt}}{\phi}+(-u_t)^{\frac{1}{k-l}}\sum_{i=1}^{n}f_{ii}\frac{\phi^{\prime}\sum_{j=1}^{n}u_ju_{iij}}{\phi}\\
&=&\frac{\phi^{\prime}}{\phi}\sum_{i,j=1}^{n}u_j\left[\frac{u_{jt}}{(k-l)u_t}+(-u_t)^{\frac{1}{k-l}}f_{ii}u_{iij}\right]=0\ \mbox{at}\ (x_0,t_0).
\end{eqnarray*}
Consequently, by \eqref{2.14-1} and \eqref{eq4},
\begin{eqnarray}
&&(-u_t)^{\frac{1}{k-l}}\sum_{i=1}^{n}f_{ii}\left[4\left(\frac{u_{ii}}{u}-\frac{u_i^2}{u^2}\right)+\frac{\phi^{\prime\prime}u_i^2u_{ii}^2}{\phi}+\frac{\phi^{\prime}u_{ii}^2}{\phi}-\frac{\phi^{\prime 2}u_i^2u_{ii}^2}{\phi^2}-\frac{u_{11i}^2}{u_{11}^2}\right]\nonumber\\
&&+(-u_t)^{\frac{1}{k-l}}\frac{1}{k-l}\sum_{i,j=1}^n\left(\frac{S_k}{S_l}\right)^{\frac{1}{k-l}-1}\frac{S_{k-2;ij}S_l-S_kS_{l-2;ij}}{S_l^2}\frac{u_{1ij}^2}{u_{11}}+\frac{4}{(k-l)u}\nonumber\\
\label{6}&\leq& 0.
\end{eqnarray}

From now on, we discuss the case of $k=n$.
Substituting
$$\frac{u_{111}}{u_{11}}=-\left(\frac{\phi_1}{\phi}+\frac{4u_1}{u}\right),\frac{u_i}{u}=-\frac{1}{4}\left(\frac{\phi_i}{\phi}+\frac{u_{11i}}{u_{11}}\right)$$
into \eqref{6}, we have
\begin{eqnarray}
0&\geq & (-u_t)^{\frac{1}{n-l}}f_{11}\left[4\left(\frac{u_{11}}{u}-\frac{u_1^2}{u^2}\right)+\frac{\phi^{\prime\prime}}{\phi}u_1^2u_{11}^2+\frac{\phi^{\prime}}{\phi}u_{11}^2-\frac{\phi^{\prime 2}}{\phi^2}u_1^2u_{11}^2-\left(\frac{\phi_1}{\phi}+\frac{4u_1}{u}\right)^2\right]\nonumber\\
&&+(-u_t)^{\frac{1}{n-l}}\sum_{i=2}^nf_{ii}\left[\frac{4u_{ii}}{u}-\frac{1}{4}\left(\frac{\phi_i}{\phi}+\frac{u_{11i}}{u_{11}}\right)^2+\frac{\phi^{\prime\prime}}{\phi}u_i^2u_{ii}^2+\frac{\phi^\prime}{\phi}u_{ii}^2
-\frac{\phi^{\prime 2}}{\phi^2}u_i^2u_{ii}^2-\frac{u_{11i}^2}{u_{11}^2}\right]\nonumber\\
&&+(-u_t)^{\frac{1}{n-l}}\frac{1}{n-l}\sum_{i,j=1}^n\left(\frac{S_n}{S_l}\right)^{\frac{1}{n-l}-1}\frac{S_{n-2;ij}S_l-S_nS_{l-2;ij}}{S_l^2}\frac{u_{1ij}^2}{u_{11}}+\frac{4}{(n-l)u}\nonumber\\
&\geq &\left\{(-u_t)^{\frac{1}{n-l}}\sum_{i=1}^nf_{ii}\left[\frac{4u_{ii}}{u}+\left(\frac{\phi^{\prime\prime}}{\phi}-\frac{3\phi^{\prime 2}}{\phi^2}\right)u_i^2u_{ii}^2+\frac{\phi^{\prime}}{\phi}u_{ii}^2\right]
-36(-u_t)^{\frac{1}{n-l}}f_{11}\frac{u_1^2}{u^2}\right\}\nonumber\\
&&+\left[-\frac{3}{2}(-u_t)^{\frac{1}{n-l}}\sum_{i=2}^nf_{ii}\frac{u_{11i}^2}{u_{11}^2}
+(-u_t)^{\frac{1}{n-l}}\frac{1}{n-l}\sum_{i,j=1}^n\left(\frac{S_n}{S_l}\right)^{\frac{1}{n-l}-1}\frac{S_{n-2;ij}S_l-S_nS_{l-2;ij}}{S_l^2}\frac{u_{1ij}^2}{u_{11}}\right]\nonumber\\
&&+\frac{4}{(n-l)u}\nonumber\\
\label{16-1}&=:&I_1+I_2+\frac{4}{(n-l)u}
\end{eqnarray}
 where
$$I_1=(-u_t)^{\frac{1}{n-l}}\sum_{i=1}^nf_{ii}\left[\frac{4u_{ii}}{u}+\left(\frac{\phi^{\prime\prime}}{\phi}-\frac{3\phi^{\prime 2}}{\phi^2}\right)u_i^2u_{ii}^2+\frac{\phi^{\prime}}{\phi}u_{ii}^2\right]
-36(-u_t)^{\frac{1}{n-l}}f_{11}\frac{u_1^2}{u^2},$$
$$I_2=-\frac{3}{2}(-u_t)^{\frac{1}{n-l}}\sum_{i=2}^nf_{ii}\frac{u_{11i}^2}{u_{11}^2}
+(-u_t)^{\frac{1}{n-l}}\frac{1}{n-l}\sum_{i,j=1}^n\left(\frac{S_n}{S_l}\right)^{\frac{1}{n-l}-1}\frac{S_{n-2;ij}S_l-S_nS_{l-2;ij}}{S_l^2}\frac{u_{1ij}^2}{u_{11}},$$
and we used \eqref{5} and the inequalities
$$\left(\frac{\phi_1}{\phi}+\frac{4u_1}{u}\right)^2\leq 2\frac{\phi_1^2}{\phi^2}+\frac{32u_1^2}{u^2},$$
and
\begin{eqnarray*}\frac{1}{4}\left(\frac{\phi_i}{\phi}+\frac{u_{11i}}{u_{11}}\right)^2&\leq& \frac{1}{2}\frac{\phi_i^2}{\phi^2}+\frac{1}{2}\frac{u_{11i}^2}{u_{11}^2}\\
&\leq &2\frac{\phi_i^2}{\phi^2}+\frac{1}{2}\frac{u_{11i}^2}{u_{11}^2}.
\end{eqnarray*}

Due to \eqref{ele-1}, then
\begin{eqnarray}
\sum_{i=1}^{n}f_{ii}u_{ii}&=&\sum_{i=1}^{n}\mu^{\prime}\frac{S_{n-1;i}S_l-S_nS_{l-1;i}}{S_l^2}\lambda_i\nonumber\\
&=&\frac{\mu^{\prime}}{S_l^2}\left(\sum_{i=1}^{n}S_l S_{n-1;i}\lambda_i-\sum_{i=1}^{n}S_nS_{l-1;i}\lambda_i\right)\nonumber\\
&=&\frac{\mu^{\prime}}{S_l^2}(nS_nS_l-lS_nS_l)\nonumber\\
\label{e2}&=&(n-l)\mu^{\prime}\frac{S_n}{S_l}=\left(\frac{S_n}{S_l}\right)^{\frac{1}{n-l}}=f(D^2u)=(-u_t)^{-\frac{1}{n-l}}.
\end{eqnarray}
By \eqref{e2} and $\frac{\phi^{\prime\prime}}{\phi}-\frac{3\phi^{\prime 2}}{\phi^2}\geq 0$ and $\phi^{\prime}\geq 0,$
then $I_1$ can be estimated
\begin{eqnarray}
I_1&\geq& \frac{4}{u}+(-u_t)^{\frac{1}{n-l}}\theta_8f_{11}u_{11}^2-C_2(-u_t)^{\frac{1}{n-l}}\frac{f_{11}}{u^2},
\end{eqnarray}
where $\theta_8=\theta_8(\phi)$ and $C_2=C_2(||u||_{C^1(D)})$. We can assume that $\Phi(x_0,t_0)$ is large such that
\begin{equation}\label{16-3}u(x_0,t_0)^2u_{11}(x_0,t_0)^2\geq \frac{2C_2}{\theta_8},\end{equation}
as otherwise, $\Phi(x_0,t_0)\leq C,$ then \eqref{sec-est} is obvious.
So in the following, we always assume that \eqref{16-3} holds. Then
\begin{eqnarray}
I_1&\geq & (-u_t)^{\frac{1}{n-l}}\frac{1}{2}\theta_8f_{11}u_{11}^2+\frac{4}{u}.\label{16-2}
\end{eqnarray}

 $I_2$ can be estimated by \eqref{j-2.18-1},
\begin{eqnarray}
I_2&\geq &-\frac{3}{2}(-u_t)^{\frac{1}{n-l}}\frac{1}{n-l}\left(\frac{S_n}{S_l}\right)^{\frac{1}{n-l}-1}\sum_{i=2}^n\frac{S_{n-1;i}S_l-S_nS_{l-1;i}}{S_l^2}\frac{u_{11i}^2}{u_{11}^2}\nonumber\\
&&+2(-u_t)^{\frac{1}{n-l}}\frac{1}{n-l}\left(\frac{S_n}{S_l}\right)^{\frac{1}{n-l}-1}\sum_{i=2}^n\frac{S_{n-2;1i}S_l-S_nS_{l-2;1i}}{S_l^2}\frac{u_{11i}^2}{u_{11}}\nonumber\\
&=&2(-u_t)^{\frac{1}{n-l}}\frac{1}{n-l}\left(\frac{S_n}{S_l}\right)^{\frac{1}{n-l}-1}\nonumber\\
&&\sum_{i=2}^n\left(\frac{S_{n-2;1i}S_l-S_nS_{l-2;1i}}{S_l^2}-\frac{3}{4}\frac{S_{n-1;i}S_l-S_nS_{l-1;i}}{\lambda_1S_l^2}\right)\frac{u_{11i}^2}{\lambda_1}.\label{2.27*}
\end{eqnarray}
Then by \eqref{2.27*} and \eqref{ell-1},
\begin{equation}\label{18-1}I_2\geq 0.\end{equation}
By \eqref{16-2} and \eqref{18-1}, then \eqref{16-1} can be converted to
\begin{equation}\label{2.28-1} 0\geq (-u_t)^{\frac{1}{n-l}}\frac{1}{2}\theta_8f_{11}u_{11}^2+\frac{4}{u}\left(1+\frac{1}{n-l}\right).\end{equation}
Multiplying $(-u)^4\phi$ on both sides in \eqref{2.28-1}, we have
\begin{equation}\label{18-2}
0\geq \frac{1}{2}\theta_8(-u_t)^{\frac{1}{n-l}}f_{11}u_{11}^2(-u)^4\phi-C_3(-u)^3\phi,
\end{equation}
where $C_3=4(1+\frac{1}{n-l})$.

Due to \eqref{j-f11-3},
\begin{eqnarray}\label{f11-4}f_{11}u_{11}^2&\geq& \mu^{\prime}\frac{n-l}{n}\frac{\lambda_1S_{n-1}}{S_l}u_{11}\nonumber\\
&\geq &\mu^{\prime}\frac{n-l}{n}\frac{S_{n}}{S_l}u_{11}\nonumber\\
&=&\frac{1}{n}(-u_t)^{-\frac{1}{n-l}}u_{11}.
\end{eqnarray}

Consequently, by \eqref{18-2}, we have that
$$0\geq \frac{\theta_8}{2n}(-u)^4\phi u_{11}-C(-u)^3\phi.$$
Thus
\begin{eqnarray*}
\Phi=(-u)^4\phi u_{11}&\leq& \frac{2nC(-u)^3\phi}{\theta_8}\\
&\leq & C_0(n,l,||u||_{C^1(D)}).
\end{eqnarray*}
Then $(-u)^4|D^2u|$ can be estimated.
\end{proof}

\section{Proof of Theorem \ref{thm}}

In this section, we give the proof of Theorem \ref{thm} using the interior gradient estimate and Pogorelov type estimate.

Let $u\in C^{4,2}(\mathbb{R}^n\times(-\infty,0])$ be a parabolically convex solution to \eqref{hq1} which satisfies \eqref{ut} and \eqref{u0}. Without loss of generality, we assume that $u(x,0)\geq 0, u(0,0)=0$.

\begin{lemma}\label{lem3}
Let \eqref{u0} hold. Then there exists a constant $A_2>0$ depending only on $A,m_2$ and $n$ such that
\begin{equation}\label{hq-7}
u(x,0)\geq A_2|x|^2,\ \ x\in \mathbb{R}^n.
\end{equation}
If $l=n-1$, this holds for $A_2=\frac{1}{2m_2}$ without assumption \eqref{u0}.
\end{lemma}

\begin{proof}
By \eqref{hq1} and the Newton-Maclaurin inequality
$$\left(\frac{S_l(\lambda)}{C_n^l}\right)^{\frac{1}{l}}\geq (S_n(\lambda))^{\frac{1}{n}},$$
where $C_n^l=\frac{n!}{l!(n-l)!}$, we have
$$S_n^{\frac{1}{l}-\frac{1}{n}}\geq \left(\frac{C_n^l}{-u_t}\right)^{\frac{1}{l}}.$$
So
$$S_n\geq \left(\frac{C_n^l}{m_2}\right)^{\frac{n}{n-l}}:=\tilde{m}.$$
Then
\begin{equation}\label{ma-1}
\det D^2u(x,0)\geq \tilde{m},\  \ x\in \mathbb{R}^n.
\end{equation}

For $L>B$, let $\Omega_0=\{x\in \mathbb{R}^n:u(x,0)<L-B\}$. Since $u(x,0)$ is strictly convex, then $\Omega_0$ is a nonempty convex open set. Let $\Gamma$ be the ellipsoid of smallest volume containing $\Omega_0$. By John's Lemma,
$$\Gamma'\equiv \frac{1}{n}\Gamma\subset \Omega_0\subset \Gamma.$$
Therefore $u(x,0)\leq L-B$ in $\o{\Gamma'}$.

Let $\Gamma'$ be defined by
$$\sum_{i=1}^n\frac{(x_i-x_i^0)^2}{a_i^2}\leq 1,\ a_1\geq \cdots\geq a_n>0,$$
where $x^0=(x_1^0,\dots,x_n^0)\in \mathbb{R}^n.$ Consider the function
$$v(x,0)=\frac{\tilde{m}^{\frac{1}{n}}}{2}(a_1\cdots a_n)^{\frac{2}{n}}\left(\sum_{i=1}^n\frac{(x_i-x_i^0)^2}{a_i^2}-1\right),\ \ x\in \mathbb{R}^n.$$
Then $\det D^2v(x,0)=\tilde{m}$ and $v=0$ on $\partial \Gamma'$. So, $u(x,0)-L+B\leq v(x,0)$ in $\o{\Gamma'}$ by the comparison principle. Particularly,
\begin{equation}\label{ma-2}L-B\geq L-B-u(x^0,0)\geq -v(x^0,0)=\frac{\tilde{m}^{\frac{1}{n}}}{2}(a_1\cdots a_n)^{\frac{2}{n}}.\end{equation}
By \eqref{u0}, we have $B_{\rho}(0)\subset \Omega_0$ where $\rho=\sqrt{(L-B)/A_1}$. Therefore, $\rho\leq n a_i$ for any $i$. Then from \eqref{ma-2}, we have that
\begin{equation}\label{ma-3}
a_1\leq \frac{(2(L-B))^{\frac{n}{2}}}{\tilde{m}^{\frac{1}{2}}a_2\cdots a_n}\leq \frac{(2(L-B))^{\frac{n}{2}}}{(\rho/n)^{n-1}}\leq C_0\sqrt{L-B},
\end{equation}
where $C_0$ depends only on $A_1,m_2,n$ and $l$. Notice that $0\in \Omega_0$, so $|x|\leq 2na_1$ for any $x\in \Omega_0$. We get from \eqref{ma-3} that
$$u(x,0)=L-B\geq \left(\frac{a_1}{C_0}\right)^2\geq A_2|x|^2,\ \ x\in \partial \Omega_0$$
for some positive constant $A_2$ depending only on $A_1,m_2,n,$ but not on $L$. Note that the level set $\Omega_0$ depends on the level $L-B$ which is arbitrary. This proves \eqref{hq-7}.

Finally, if $l=n-1$, then by \eqref{hq1}, the eigenvalues of $D^2u$, $\lambda_i(D^2u)\geq 1/m_2$ everywhere for $ i=1,\dots,n$, and thus $u(x)\geq \frac{1}{2m_2}|x|^2.$
\end{proof}

For $0<\alpha<1$ and a domain $\Omega\subset\mathbb{R}^n\times(-\infty,0],$ let 
$$[u]_{C^{\alpha}(\o{\Omega})}=\sup_{\substack {(x,t),(y,s)\in \Omega\\(x,t)\not=(y,s)}}\frac{|u(x,t)-u(y,s)|}{(|x-y|^2+|t-s|)^{\frac{\alpha}{2}}},$$
$$||u||_{C^{2+\alpha,1+\frac{\alpha}{2}}(\o{\Omega})}=\sum_{2i+j\leq 2}||D^i_tD^j_x u||_{C^{0}(\o{\Omega})}+[D^2u]_{C^{\alpha}(\o{\Omega})}+[u_t]_{C^{\alpha}(\o{\Omega})},$$
where $D^i_tD^j_x u$ denotes the $j-$th order derivative with respect to $x$ and $i-th$ order derivative with respect to $t$ of $u(x,t)$, $i,j\geq0$ are integers.

\begin{proof}[Proof of Theorem \ref{thm}]
Let
\begin{equation}\label{r0}
R_0^2=2B.
\end{equation}
Define
$$U_R=\{(x,t)\in \mathbb{R}^n\times(-\infty,0]|u(Rx,R^2t)<R^2\},$$
and
$$U_R(t)=\{x|(x,t)\in U_R\}.$$
Since $u$ is decreasing in $t$, then $U_R(t)$ is increasing in $t$, that is, for any $t_1\leq t_2$, we have that
$$U_{R}(t_1)\subset U_{R}(t_2).$$
So $U_R(t)$ is a bowl-shaped domain.

Since $u(x,0)\geq A_2|x|^2$, then
$$U_R(t)\subset U_R(0)\subset B_{\frac{1}{\sqrt{A_2}}}.$$
By \eqref{ut}, then for $(x,t)\in U_R$, we have
$$-m_1R^2t\leq A_2|Rx|^2-m_1R^2t\leq u(Rx,0)-m_1R^2t\leq u(Rx,R^2t)<R^2.$$
Thus
$$t\geq -\frac{1}{m_1}.$$
So
$$U_R\subset B_{1+\frac{1}{\sqrt{A_2}}}\times(-1-\frac{1}{m_1},0].$$
Define for $R>R_0$,
$$w(x,t)=\frac{u(Rx,R^2t)-R^2}{R^2}.$$
Then $w\in C^{4,2}(\mathbb{R}^n\times (-\infty,0])$ is parabolically convex,
$$w_t=u_t(Rx,R^2t), D^2w=D^2u(Rx,R^2t),$$
\begin{align}\label{5-1}m_1\leq -w_t\leq m_2\ \ \mbox{in}\ \ \mathbb{R}^n\times(-\infty,0],\end{align}
and $w$ satisfies
$$\begin{cases}
-w_t\dfrac{S_n(D^2u)}{S_l(D^2w)}=1\ \ &\mbox{in}\ \ U_R,\\
w=0\ \ &\mbox{on}\  \ \partial_p U_R.
\end{cases}$$
Moreover, by \eqref{r0} and Lemma \ref{lem3},
\begin{align}
\nonumber A_1|x|^2-\frac{1}{2}&\geq A_1|x|^2+\frac{B}{R_0^2}-1\geq\frac{A_1R^2|x|^2+B}{R^2}-1\\
\nonumber&\geq w(x,0)=\frac{u(Rx,0)-R^2}{R^2} \geq \frac{A_2R^2|x|^2}{R^2}-1=A_2|x|^2-1.
\end{align}
That is,
\begin{equation}\label{5-2}A_1|x|^2-\frac{1}{2}\geq  w(x,0)\geq A_2|x|^2-1.\end{equation}

In virtue of \eqref{5-1} and \eqref{5-2}, we know that in $B_{1+\frac{1}{\sqrt{A_2}}}\times(-1-\frac{1}{m_1},0],$
\begin{align}\nonumber-1\leq A_2|x|^2-1-m_1t\leq w(x,0)-m_1t\leq w(x,t)\\
\label{6-1}\leq w(x,0)-m_2t\leq A_1|x|^2-\frac{1}{2}-m_2t\leq C,
\end{align}
where $C=C(A_2,m_1,m_2)$ is a constant. By \eqref{6-1} and Theorem \ref{gradient}, we know that
$$|Dw|\leq C\ \ \mbox{in}\  \ U_R,$$
where $C=C(n,l,m_1)$.

For $\tau<0$, let
$$\Omega_\tau=\{(x,t)\in U_R|w(x,t)<\tau\}.$$
Choose $\tau=-\frac{1}{4}$, then from Theorem \ref{thm-ns}, we have that
$$\sup_{\Omega_{-\frac{1}{4}}}|w(x,t)|^4|D^2w|\leq C,$$
where $C$ depends only on $n,m_1,m_2.$
By \eqref{6-1}, we can get
$$|D^2w|\leq C\ \ \mbox{in}\ \ \Omega_{-\frac{1}{3}},$$
where $C=C(n,l,m_1,m_2)$.

Let $V=\{(x,t)\in\mathbb{R}^n\times(-\infty,0]||x|<\frac{1}{\sqrt{8A_1}},t>-\frac{1}{48m_2}\}.$ Due to \eqref{5-1} and \eqref{5-2}, we obtain that for any $(x,t)\in V$,
$$w(x,t)\leq w(x,0)-m_2t\leq A_1|x|^2-\frac{1}{2}-m_2t<-\frac{1}{3}.$$
So $V\subset \Omega_{-\frac{1}{3}}$. By Theorem 4.2 in \cite{nt}, we know that for some $0<\alpha<1$,
$$||w||_{C^{2+\alpha,1+\frac{\alpha}{2}}(\o{V'})}\leq C,$$
where $V'=\{(x,t)\in\mathbb{R}^n\times(-\infty,0]||x|<\frac{1}{\sqrt{10A_1}},t>-\frac{1}{50m_2}\}.$ Then
$$[D^2u]_{C^{\alpha}(\o{V''})}=R^{-\alpha}[D^2w]_{C^{\alpha}(\o{V})}\leq CR^{-\alpha},$$
$$[u_t]_{C^{\alpha}(\o{V''})}=R^{-\alpha}[w_t]_{C^{\alpha}(\o{V})}\leq CR^{-\alpha},$$
where $V''=\{(x,t)\in\mathbb{R}^n\times(-\infty,0]||x|<\frac{R}{\sqrt{10A_1}},t>-\frac{R^2}{50m_2}\}.$
Thus we have as $R\to+\infty$,
$$[D^2u]_{C^{\alpha}(\mathbb{R}^n\times(-\infty,0])}=0,$$
$$[u_t]_{C^{\alpha}(\mathbb{R}^n\times(-\infty,0])}=0.$$
As a result, $u$ has the desired form. We complete the proof.
\end{proof}

Similar to the proof of Theorem \ref{thm}, by Theorem \ref{thm-hqe} and the gradient estimate in \cite{chen}, we can also obtain the Bernstein theoremof elliptic Hessian quotient equation \eqref{hq-3} in $\mathbb{R}^n$. See also Theorem 1.2 in \cite{BCGJ}.

\begin{theorem}\label{ell-thm}
Let $u\in C^{4}(\mathbb{R}^n)$ be a strictly convex solution to
$$\frac{S_{n}(D^2u)}{S_{l}(D^2u)}=1\
\ \mbox{in}\ \ \mathbb{R}^n.$$
Assume that there exist positive constants $\tilde{A}_1, \tilde{B}$ such that for $x\in \mathbb{R}^n$,
\begin{equation*}\label{ell-u0}u(x)\leq \tilde{A}_1|x|^2+\tilde{B}.\end{equation*}
 Then $u$ must be a convex quadratic polynomial.
\end{theorem}

(L.M. Dai)  School of Mathematics and Information Science, Weifang
 University, Weifang, 261061, P. R. China

 {\it Email address}: lmdai@wfu.edu.cn
 \vspace{3mm}

(J.G. Bao)School of Mathematical Sciences, Beijing Normal University,
Laboratory of Mathematics and Complex Systems, Ministry of
Education, Beijing, 100875, P. R. China

{\it Email address}: jgbao@bnu.edu.cn
\vspace{3mm}

(B. Wang)School of Mathematics and Statistics, Beijing Institute of Technology, Beijing, 100081, P. R. China

{\it Email address}: wangbo89630@bit.edu.cn

\end{document}